\documentclass[12pt,a4paper]{amsart}
\usepackage{amsmath,amsthm}
\usepackage{amssymb}
\usepackage[a4paper,centering,twoside,textwidth=6in]{geometry}
\usepackage[arrow,matrix]{xy}
\usepackage{amsmath,amsfonts,amssymb,amscd,verbatim,delarray,verbatim}

\theoremstyle{plain}
\newtheorem{theorem}{Theorem}[section]

\newtheorem{thm}{Theorem}[section]
\newtheorem{lem}[thm]{Lemma}
\newtheorem{cor}[thm]{Corollary}
\theoremstyle{definition}
\newtheorem{ex}[thm]{Example}
\newtheorem{rem}[thm]{Remark}

\theoremstyle{definition}
\newtheorem{defn}[theorem]{Definition}

\newcommand{\thmref}[1]{Theorem~\ref{#1}}
\newcommand{\secref}[1]{Section~\ref{#1}}

 \DeclareMathOperator{\GL}{GL}
 
\newcommand{\BZ}{\mathbb{Z}}
\newcommand{\BC}{\mathbb{C}}

\newcommand{\BP}{\mathbb{P}}
\newcommand{\BQ}{\mathbb{Q}}

\DeclareMathSymbol{\twoheadrightarrow}  {\mathrel}{AMSa}{"10}

\def\P{{\mathbb P}}

\def\A8{{\mathbf A}_8}
\def\Bir{\mathrm{Bir}}

\def\End{\mathrm{End}}
\def\Aut{\mathrm{Aut}}

                                   \def\rank{\mathrm{rk}}

\def\fchar{\mathrm{char}}
             
\def\GL{\mathrm{GL}}

                                                  \def\GO{\mathrm{GO}}
                                                      \def\PO{\mathrm{PGO}}

\def\A{\mathcal{A}}

\def\OO{\mathrm{O}}
            \def\SO{\mathrm{SO}}

\def\dim{\mathrm{dim}}\def\codim{\mathrm{codim}}

                                                \def\Ac{{\mathcal A}}
                                                 \def\Bc{{\mathcal B}}

          \def\l1{{\mathbf 1}}

\title[Jordan groups]{Jordan groups, conic bundles and abelian varieties}
\thanks{The second named author is partially supported by a grant from the Simons Foundation (\#246625 to Yuri Zarkhin). Part of this work was done in May-June 2016 during his stay at the Max-Planck-Institut f\"ur Mathematik, whose hospitality and support are gratefully acknowledged.}

\author {Tatiana Bandman}

\author[Yuri G.\ Zarhin]{Yuri G.\ Zarhin}
\address{Department of
Mathematics, Bar-Ilan University, 5290002, Ramat Gan, ISRAEL}
\email{bandman@macs.biu.ac.il}
\address{Department of Mathematics, Pennsylvania State University,
University Park, PA 16802, USA}

\email{zarhin\char`\@math.psu.edu}

\begin{document}

\begin{abstract} 
 A group 
$G$ is called {\sl Jordan} if there is a positive integer $J=J_G$ such that every finite subgroup $\mathcal{B}$ of $G$ contains a commutative subgroup $\mathcal{A}\subset \mathcal{B}$ such that $\mathcal{A}$ is normal in $\mathcal{B}$ and the index $[\mathcal{B}:\mathcal{A}] \le J$ (V.L. Popov).  In this paper we deal with Jordaness properties of the groups $\Bir(X)$ of birational automorphisms  of  irreducible smooth projective varieties $X$ over an algebraically closed field of characteristic zero. It is known (Yu. Prokhorov - C. Shramov) that $\Bir(X)$ is Jordan if $X$ is {\sl non-uniruled}. On the other hand,  the second named author proved that $\Bir(X)$  is {\sl not} Jordan if $X$ is birational to a product of the projective line  $\BP^1$ and a positive-dimensional abelian variety. 

We prove that   $\Bir(X)$ is Jordan if (uniruled) $X$ is a {\sl conic bundle} over a {\sl non-uniruled} variety $Y$ but is {\sl not} birational to $Y\times\BP^1$. (Such a conic bundle exists if and only if $\dim(Y)\ge 2$.)  When $Y$ is an abelian surface, this Jordaness property result gives an answer to a question of Prokhorov and Shramov.
\end{abstract}


\maketitle

\section{Introduction}

In this paper we deal with the groups of birational and biregular automorphisms of algebraic varieties in characteristic {\sl zero}. 
 
If $X$ is an  irreducible algebraic variety over a field $K$ of characteristic zero then we write  $\mathcal{O}_X$ for the structure sheaf of $X$, $\Aut(X)=\Aut_K(X)$ (resp.  $\Bir(X)=\Bir_K(X)$) for the group of its biregular (resp. birational) automorphisms and $K(X)$ for the field of rational functions on $X$.  We have
$$\Aut(X)\subset \Bir(X)=\Aut_K(K(X))$$
where $\Aut_K(K(X))$ is the group of $K$-linear field automorphisms of $K(X)$. We write $\mathrm{id}_X$ for the identity automorphism of $X$, which is the identity element of the groups $\Aut(X)$ and $\Bir(X)$.
If $n$ is a positive integer then we write $\mathbb{P}^n_K$ (or just ${\mathbb{P}^n}$ when it does not cause a confusion) for the $n$-dimensional projective space over $K$.

If $X$ is  smooth projective then we write $q(X)$ for its {\sl irregularity}.  For example, if $X$ is an abelian variety then $q(X)=\dim(X)$. 

In what follows we write $k$ for an algebraically closed field of characteristic zero.  We write  $\cong$ 
and $\sim$ 
for an  isomorphism 
and birational isomorphism
 of  algebraic varieties
respectively.
If $\mathcal{A}$ is a finite commutative group then we call its rank the smallest possible number of its generators and denote it by $\rank(\mathcal{A})$. If $\mathcal{B}$ is a finite group then we write $\mid \mathcal{B}\mid$ for its order.

\subsection     { Jordan groups.}\label{jordan  groups}

Recall (Popov \cite{Popov1,Popov2}) the following definitions that were motivated by the classical theorem of Jordan about finite subgroups of the complex matrix group $\GL(n,\BC)$ \cite[\S 36]{CR}.

\begin{defn}\label{def:jordan}
Let $G$ be a group.
\begin{itemize}
\item
$G$ is called {\sl Jordan} if there is a positive integer $J=J_G$ such that every finite subgroup $\mathcal{B}$ of $G$ contains a commutative subgroup $\mathcal{A}\subset \mathcal{B}$ such that $\mathcal{A}$ is normal in $\mathcal{B}$ and the index $[\mathcal{B}:\mathcal{A}] \le J$. 
\item
We say that $G$ has {\sl finite subgroups of bounded rank}  if there is a positive integer $m=m_G$ such that  any finite abelian subgroup $\mathcal{A}$ of $G$ can be generated by at most $m$ elements \cite{TurullM,ProSha1}.  
\item

We call a Jordan group $G$ {\sl strongly Jordan} if there is a positive integer $m=m_G$ such that  any finite abelian subgroup $\mathcal{A}$ of $G$ can be generated by, at most, $m$ elements \cite{TurullM}. In other words, $G$ is strongly Jordan if it is Jordan and has finite 
abelian
subgroups of bounded rank.
\item
We say that $G$ is {\sl bounded}  \cite{Popov2,ProSha1} if there is a positive integer $C=C_G$ such that the order of every finite subgroup of $G$  does not exceed $C$. (A bounded group is Jordan and even strongly Jordan.)
\end{itemize}
\end{defn}




 One may introduce similar definitions for families of groups  \cite{Popov1,ProSha1}.

\begin{defn}
Let $\mathcal{G}$ be a family of groups.
\begin{itemize}
\item 
We say that  $\mathcal{G}$ is  {\sl uniformly Jordan} (resp. {\sl uniformly strongly Jordan}) if   there is a positive integer $\mathcal{J}$ (resp.  there are positive integers  $\mathcal{J}$ and $\mathcal{M}$)  such that each $G \in \mathcal{G}$  enjoys Jordan property (resp. strong Jordan property) with $J_G=\mathcal{J}$ (resp. with $J_G=\mathcal{J}$ and $m_G=\mathcal{M}$).   
\item
We say that $\mathcal{G}$ is {\sl uniformly bounded} if there is a positive integer $C$ such that  the order of every finite subgroup of every $G$ from  $\mathcal{G}$  does {\sl not} exceed $C$. (See \cite[Remark 2.9 on p. 2058]{ProSha1}.)
\end{itemize}
\end{defn}

\begin{rem}
In the terminology of \cite[p. 2067]{ProSha1}, a family  $\mathcal{G}$ is uniformly strongly Jordan if and only if it is uniformly  Jordan and has finite subgroups of uniformly bounded rank.
\end{rem}

\subsection{ Jordan properties of $\Bir(X)$ and $\Aut(X)$.}\label{subsect:1.3}

Let $X$ be an irreducible quasiprojective variety over $k$.  There is the natural group embedding
$$\Aut(X) \hookrightarrow \Bir(X)$$
that allows us to view $\Aut(X)$ as a subgroup of $\Bir(X)$. In particular, the Jordan property (resp. the strong Jordan property) for $\Bir(X)$ implies the same property for  $\Aut(X)$. However, the converse is not necessarily true.  More precisely:
\begin{itemize}\item
It  is known that $\Aut(X)$ is Jordan if $\dim(X) \le 2$ \cite{Popov1,ZarhinTG,BandmanZarhinTG}. 
\item It is also known (Popov \cite{Popov1}) that $\Bir(X)$ is Jordan if $\dim(X) \le 2$ and $X$ is {\sl not} biratrional to a product $E \times \mathbb{P}^1$ of an elliptic curve $E$ and the projective line $\mathbb{P}^1$.   (The Jordan property of the two-dimensional Cremona group  $\Bir(\mathbb{P}^2)$ was established earlier by J.-P. Serre \cite{Serre}.) 
\item
In the remaining case $\Bir(E \times \mathbb{P}^1)$ is {\sl not} Jordan. More   generally, if $A$ is an abelian variety of {\sl positive dimension} over $k$ and $n$ is a positive integer, then $\Bir(A \times \mathbb{P}^n)$ is {\sl not} Jordan \cite{ZarhinPEMS}. 

 Notice that $\Bir(A)$ coincides with $\Aut(A)$ and is strongly Jordan \cite{Popov1}.   Actually, if  $\mathbf{A}_d$ is the family of groups $\Bir(A)$  when $A$ runs through the set of all $d$-dimensional abelian varieties over $k$ then  $\mathbf{A}_d$ is uniformly strongly Jordan \cite[Corollary 2.15 on p.  2058]{ProSha1}.
\item
In higher dimensions, a recent result of Meng and Zhang \cite{MengZhang} asserts that $\Aut(X)$ is Jordan if $X$ is {\sl projective}.
\end{itemize}

  For groups of birational automorphisms  in higher dimensions Prokhorov and Shramov \cite{ProSha1} proved the following strong result \cite[Th.  1.8]{ProSha1}.

\begin{thm}\label{PS}{                                                                                                                                                  }
{                                           }
\begin{enumerate}
\item
If $X$ is  non-uniruled then $\Bir(X)$ is Jordan.
\item
If $X$ is  non-uniruled and $q(X)=0$ then $\Bir(X)$ is bounded. 
\end{enumerate}
\end{thm}

 \begin{rem}
\label{PSr}
Prokhorov and Shramov \cite[Remark 6.9 on p. 2065]{ProSha1} noticed that if $X$ is  non-uniruled then $\Bir(X)$ has finite subgroups of bounded rank. This means that in the non-uniruled case $\Bir(X)$ is {\sl strongly Jordan}.
\end{rem}

In addition, in dimension $3$ Prokhorov and Shramov  proved \cite{ProSha2,ProSha1} that:

\begin{itemize}
\item
 If   $q(X)=0$ then $\Bir(X)$ is  {\sl Jordan}.
\item
If $X$ is rationally connected then $\Bir(X)$ is strongly Jordan. Even better, 
 if $X$ varies in the set of rationally connected threefolds then the corresponding family of groups $\Bir (X)$ is {\sl uniformly strongly Jordan} \cite[Th. 1.7 and 1.10]{ProSha2}.   
\end{itemize}

 Actually,  they proved all these assertions in arbitrary dimension $d$, assuming that the well-known conjecture of A. Borisov, V. Alexeev and L. Borisov (BAB conjecture \cite{Bor}) about the boundedness of  families of $d$-dimensional Fano varieties with terminal singularities is valid in dimension $d$. 

In light of their results it remains to investigate Jordaness properties of $\Bir(X)$ when $X$ is {\sl uniruled} with $q(X)>0$. According to Prokhorov and Shramov
 \cite[p. 2069]{ProSha1}, it is natural to start with a    {\sl conic bundle} $X$ over an {\sl abelian surface} $A$ when $X$ is {\sl not} birational to a product $A\times \mathbb{P}^1$.   

In this work we prove that in this case  $\Bir(X)$ is Jordan. Actually, we prove the following more general statement.

\begin{thm}
\label{BirX}
Let $X$ be an irreducible  smooth projective variety of dimension $d \ge 3$ over $k$ and $f: X \to Y$ be a surjective morphism over $k$ from $X$ to a $(d-1)$-dimensional abelian variety $Y$ over $k$.   Let us assume that the generic fiber of $f$ is a genus zero smooth projective irreducible curve $\mathcal{X}_f$ over $k(Y)$ without   $k(Y)$-points. Then $\Bir(X)$ is strongly Jordan.
\end{thm}

We deduce Theorem \ref{BirX} from the following  more general statement.

\begin{thm}
\label{BirXY}
Let $d \ge 3$ be a positive integer, $X$ and $Y$ are smooth  irreducible projective varieties over $k$ of dimension $d$ and $d-1$ respectively. Let $f: X \to Y$ be a surjective morphism, 
whose generic fiber  is a genus zero smooth projective irreducible curve $\mathcal{X}_f$ over $k(Y)$ without   $k(Y)$-points.
Assume that $Y$ is non-uniruled. 
Then $\Bir(X)$ is  strongly Jordan. 
\end{thm}

The following assertion is a variant of Theorem \ref{BirXY}

\begin{thm}
\label{BirXYbis}
Let $d \ge 3$ be a positive integer, $X$ and $Y$ are smooth  irreducible projective varieties over $k$ of dimension $d$ and $d-1$ respectively. Let $f: X \to Y$ be a surjective morphism. Suppose that there exists a nonempty open subset $U$ of $X$ such that for all $y \in U(k)$ the corresponding fiber $X_y$ of $f$ is $k$-isomorphic  to the projective line over $k$  (i.e. the general fiber $X_y\cong\BP^1_k$).

Assume that $Y$ is non-uniruled and $f$ does not admit a rational section $Y \dasharrow X$.
Then $\Bir(X)$ is  strongly Jordan. 
\end{thm}

The paper is organized as follows.   \secref{2}  deals with  Jordaness   of groups. In \secref{3}
we remind basic properties of conic bundles over non-uniruled varieties. 
 In section (4) we  describe finite subgroups of automorphisms group of a conics without rational points.
We prove main results of the paper  in \secref{5}. 

{\bf Acknowledgements}. We are deeply grateful to the referee, whose comments helped to improve the exposition; especially, for the suggestions to include the discussion of automorphism groups of even-dimensional quadrics over function fields (see the end of Section  \ref{linalg}) and the case of non-smooth varieties (see Remark \ref{singularConics}).

\section{Group Theory}\label{2}

We will need the following useful result of Anton Klyachko \cite[Lemma 2.8 on p.2057]{ProSha1}. (See also \cite[Lemma 6.2]{ZarhinTG}.)

\begin{lem}
\label{klyachko}
Let $\mathcal{G}_1$ and $\mathcal{G}_2$ be families of groups such that 
$\mathcal{G}_1$ is uniformly bounded  and $\mathcal{G}_2$ is uniformly strongly Jordan. Let $\mathcal{G}$ be a family of groups $G$ such that $G$ sits in an exact sequence
$$\{1\} \to G_1 \to G \to G_2 \to \{1\},$$
where $G_1 \in \mathcal{G}_1$ and $G_2 \in \mathcal{G}_2$.  Then $\mathcal{G}$ is uniformly Jordan.
\end{lem}

Actually, we will use the following slight refinement of Lemma \ref{klyachko}.

\begin{lem}
\label{klyachko2}
In the notation and assumptions of Lemma \ref{klyachko} the family $\mathcal{G}$ is uniformly strongly Jordan.
\end{lem}

\begin{proof}
The assertion follows readily from Lemma \ref{klyachko} combined with Lemma 2.7 of \cite[p. 2057]{ProSha1} about extensions of groups with uniformly bounded ranks.
\end{proof}

\begin{cor}
\label{abelianUniform}
Let $d$ be a positive integer.
Let $\mathcal{G}_1$ be a uniformly bounded family of groups.  Let $\mathbf{A}_d$ be the family of groups $\Bir(A)$ where $A$ runs through the set of all $d$-dimensional abelian varieties over $k$. Let $\mathcal{G}$ be a family of groups $G$ such that there exists an exact sequence
$$\{1\} \to G_1 \to G \to G_2 \to \{1\},$$
where $G_1 \in \mathcal{G}_1$ and $G_2 \in \mathbf{A}_d$.  Then $\mathcal{G}$ is uniformly Jordan.
\end{cor}

\begin{proof}
One has only to recall that $\mathcal{G}_2:=\mathbf{A}_d$ is uniformly strongly Jordan \cite[Corollary 2.15 on p. 2058]{ProSha1} and apply Lemma \ref{klyachko2}.
\end{proof}

\begin{lem}
\label{bound}
Let $G$ be a strongly Jordan group and let $H$ be a subgroup of $G$.
 Suppose that there exists a positive integer $N_H$ such that every periodic element  in $H$ has order that does not exceed $N_H$.

Then $H$ is bounded.
\end{lem}

\begin{proof}
 Let $J_G$ be the Jordan index of $H$. We know that there is a positive integer $m_G$ such that
every finite abelian subgroup in $G$ is generated by, at most, $m_G$ elements.

 Let $\Bc$ be a finite subgroup  of $H$. Clearly, $\Bc$ is a sugroup of $G$ as well.
Then $\Bc$ contains a finite abelian subgroup $\Ac$ with index
$[\Bc:\Ac]\le J_G$. The abelian  group $\Ac$ is generated by, at most, $m_G$ elements, each of which has order $\le N_H$. This implies that
$|\Ac| \le N_H^{m_G}$ and therefore 
$$|\Bc| \le J_H \cdot |\Ac| \le J_G \cdot N_H^{m_G}.$$
\end{proof}

Recall \cite{TurullM} that  the matrix group $\GL(n,\BC)$ is {\sl strongly} Jordan and its every finite abelian subgroup
is generated by, at most, $n$ elements.
 This implies that 
for any field $K$ of characteristic zero the matrix group
$\GL(n,K)$ is a strongly Jordan group
  with Jordan index $J_{\GL(n,\BC)}$ (see \cite[Sect. 1.2.2 on p. 187]{Popov2}); in addition, its every finite abelian subgroup
is generated by, at most, $n$ elements.
 Combining this observation with Lemma \ref{bound} and the last formula of its proof, 
we obtain the following assertion that may be of independent interest.

\begin{thm}
\label{boundA}
Let   $K$ be a field  of characteristic zero and $n$  a positive integer.
 Suppose that  $H$ is a subgroup of $\GL(n,K)$ and $N$ is a positive integer
such that   every periodic element in $H$ has order $\le N$.  Then there exist 
a positive integer $\mathbf{N}=\mathbf{N}(n,N)$ that depends only on $n$ and $N$,
and such that every finite subgroup in $H$ has order $\le \mathbf{N}$.
In particular,  $H$ is bounded.
\end{thm}


\section{Conic bundles}\label{3}
\label{conic}

Let $f: X \to Y$ be a surjective morphism  of smooth irreducible projective varieties of positive dimension over  $k$. Since $X$ and $Y$ are projective, $f$ is a {\sl projective} morphism. 
 It is well known \cite[Ch. III, Sect. 10, Cor. 10.7]{H} that there is an open Zariski dense  subset $U=U(f)$ of $Y$ such that the restriction $f^{-1}(U)\overset{f}{\to} U$ is smooth  (\cite[Ch. III, Sect. 10, Cor. 10.7]{H})  and  flat    (\cite [Lect. 8,  $2^{\circ}$] {MumfordLC}, \cite[Theorem 6.9.1] {EGA42}). Thus   the {\sl generic fiber} $\mathcal{X}:=\mathcal{X}_f$ is a smooth projective  variety over $k(Y)$ and all its irreducible components have dimension $\dim(X)-\dim(Y)$.
 (\cite[Ch. III, Sect. 9, Corollary 9.6]{H}), (\cite[Ch. III, Sect. 10, Prop. 10.1]{H}). In addition, if $y$ is a closed point of $U$ then the corresponding fiber $X_y$ of $f$ is a smooth projective variety over the field $k(y)=k$ and all its irreducible components have dimension $\dim(X)-\dim(Y)$.

Notice that {\sl dominant} $f$ defines the field embedding
$$f^{*}: k(Y) \hookrightarrow k(X)$$
that is the identity map on $k$. Further we will identify $k(Y)$ with its image in $k(X)$.
 The field of rational functions  of $\mathcal{X}_f$  coincides with $k(X)$ and the group of birational automorphisms $\Bir_{k(Y)}(\mathcal{X}_f)$ coincides with (sub)group 
\begin{equation}\label{kernel}
\Aut(k(X)/k(Y))\subset \Aut(k(X)/k)=\Bir_k(X)\end{equation}
that consists of all  automorphisms of the field $k(X)$  leaving invariant every element of $k(Y)$.

We say that $X$ is a  {\sl conic bundle} over $Y$ if the generic fiber 
$\mathcal{X}:=\mathcal{X}_f$ is an absolutely irreducible genus $0$ curve over $k(Y)$. (See \cite{Sarkisov1,Sarkisov2}.) In particular,
$\dim(X)-\dim(Y)=1$ and therefore  the general fiber of $f$ is a (smooth projective) curve.

\begin{rem}\label{generic-general}
As usual, by the {\sl general  fiber} of $f$  we mean the fiber  $X_y$ of $f$ over a point $y$ from some  nonempty open subset of $Y.$
If the generic fiber is an irreducible  smooth  projective curve then there is an open nonempty subset $U$ of $Y$ such that for all closed
points $y \in U$ the corresponding
 (closed) fibers $X_y$ are irreducible smooth  projective curves over $k(y)=k$ as well (\cite[Corollary 9.5.6, Proposition 9.7.8]{EGA43}). Semi-continuity Theorem (\cite[Ch III, Theorem  12.8]{H},\cite [Corollary, p.47]{MumfordAV} implies that  the general fiber has genus  zero if and only if the
same is true for the generic fiber.  Thus, the condition that the generic fiber is a smooth irreducible curve of genus zero is in our setting equivalent to the same condition for the general fiber.
\end{rem}

\begin{rem}
\label{productYP}
If the genus $0$ curve  $\mathcal{X}_f$ has a $k(Y)$-rational point then  it  is biregular to the projective line over $k(Y)$ \cite[Th.  A.4.3.1 on p. 75]{HindryS}.
 This implies that  $X$ is $k$-birational   to $Y \times \P^1$.   It follows from \cite{ZarhinPEMS}  that if $Y$ is an abelian variety (of positive dimension) then
$\Bir(X)$ is {\sl not} Jordan.
 
\end{rem}

\begin{ex}
\label{aniquad}
Let us consider a smooth projective plane quadric
$$\mathcal{X}_q=\{a_1 T_1^2 +a_2 T_2^2+a_3 T_3^2=0\}\subset{\mathbb{P}^2}_{k(Y)}$$
over the field $K:=k(Y)$
where all $a_i$ are {\sl nonzero} elements of $k(Y)$ such that the nondegenerate ternary quadratic form
$$q(T)=a_1 T_1^2 +a_2 T_2^2+a_3 T_3^2$$
in $T=(T_1,T_2,T_3)$
is {\sl anisotropic} over $k(Y)$, i.e., $q(T)\ne 0$ if  all $T_i \in k(Y)$ and, at least, one of them is not $0$ in $k(Y)$.  (It follows from \cite[Th. 1 on p. 155]{Kahn}  that such a form  exists if and only if $2^{\dim(Y)}\ge 3$, i.e., if and only if $\dim(Y)\ge 2$.)
Clearly, $\mathcal{X}_q$ is an absolutely irreducible smooth projective curve of genus $0$ over $K$ that does {\sl not} have $K$-points.

We want to construct a conic bundle with generic fiber  $\mathcal{X}_q$ 
without $K$-rational point.
First, let us consider
the field $K(\mathcal{X}_q)$ of the rational functions on $\mathcal{X}_q$. It is finitely generated  over $K$ and has transcendence degree $1$ over it. This implies that  $K(\mathcal{X}_q)$  is finitely generated over $k$ and has transcendence degree $\dim(Y)+1$ over it.  Since $\mathcal{X}_q$ is absolutely irreducible over $K$, the latter is algebraically closed in $K(\mathcal{X}_q)$.

By Hironaka's results, there is an irreducible smooth projective variety $X$ over $k$ of dimension $\dim(Y)+1$ with $k(X)=K(\mathcal{X}_q)$  such that
the dominant rational map $f:X \to Y$ induced by the field embedding $k(Y)=K \subset K(\mathcal{X}_q)=k(X)$ is actually a morphism. Clearly,
the generic fiber $\mathcal{X}_f$ is a smooth projective variety, all whose irreducible components have dimension $1$. Since
 $K$ is algebraically closed in $K(\mathcal{X}_q)=k(X)$,  the curve $\mathcal{X}_f$ is absolutely irreducible over $K$
\cite[Cor. 4.3.7 and Remark 4.3.8]{EGA31}. On the other hand, the field $K(\mathcal{X}_f)$ of rational functions on the $K$-curve 
$\mathcal{X}_f$ coincides with $k(X)$ by the very definition of the generic fiber. Since $K(\mathcal{X}_f)=k(X)=K(\mathcal{X}_q)$, the $K$-curves
$\mathcal{X}_f$ and $\mathcal{X}_q$
are birational. Taking into account that both curves are smooth projective and absolutely irreducible over $K$, we conclude that $\mathcal{X}_f$
and $\mathcal{X}_q$ are biregularly isomorphic over $K$. This implies that $\mathcal{X}_f$ has genus zero and has no  $k(Y)$-rational points.  Thus  $f:X \to Y$ is the  conic bundle  we wanted to construct.


\end{ex}

\begin{lem}\label{Kaw}
Let $X$ and $Y$ be smooth irreducible projective varieties of positive dimension over $k,$
 $f: X \to Y$ be a surjective   morphism, such that the general fiber $F_y=f^{-1}(y)$ 
is isomorphic to  $ \BP^1_k.$ Let us identify $k(Y)$ with its image in $k(X)$. 
 Assume additionally that $Y$is non-uniruled. (E.g., $Y$ is an abelian variety.)

Then every $k$-linear automorphism $\sigma$ of the field $k(X)$ leaves invariant $k(Y)$, i.e.,
$$\sigma (k(Y))=k(Y), \   \forall \sigma \in \Aut(k(X)).$$
In addition, there is exactly one birational  automorphism $u_Y$ of $Y$, whose action on $k(Y)$ coincides with $\sigma$.
\end{lem}

\begin{proof}

There is a birational automorphism $u_X$ of $X$ that induces $\sigma$ on $k(X)$.
Let $\tilde X\overset{\pi}{\to} X$ be  a resolution of indeterminancy of $u_X,$ i.e.,  
we consider a smooth irreducible  projective $k-$variety  $\tilde X,$  and birational   morphisms 
$$\pi,  \tilde u_X :\tilde{X} \to X$$
that enjoy the following properties.
\begin{itemize}\item
 $\pi^{-1}: X \dasharrow \tilde X$ is an isomorphism outside the indeterminancy locus of $u_X;$
\item
 the following diagram commutes:

\begin{equation}\label{diagram1006}\begin{aligned}
& &   &\tilde X       &  {}  & {}        &   && \notag \\
& &\pi& \downarrow  &        {\searrow}                     &  \tilde {u_X}  & &&\notag \\
& & &X       &\stackrel{u_X}{\dashrightarrow } &X  & &&  \\ 
& &f & \downarrow  &  {}&\downarrow  f& &&\notag \\
& &  &Y           &{} &Y  &&    &  
\end{aligned}.  \end{equation}
\end{itemize}

Consider morphisms $\tilde f=f\circ\pi:\tilde X\to Y$ and $ g=f\circ \tilde u_X:\tilde X\to Y.$ 

Let $\Sigma_1\subset  X,  \ \Sigma_2\subset X$ be  the loci of indeterminancy  of  $\pi^{-1}$ and $\tilde u_X^{-1}$
respectively.

 Since $\codim_X( \Sigma_1)\geq 2$  and $\codim_X(\Sigma_2)\geq 2$, we obtain that
$\codim_Y (f(\Sigma_1))\geq 1$ and $\codim_Y( f(\Sigma_2))\geq 1.$ (Recall that $\dim(X)=\dim(Y)+1$.)

This implies that there is a nonempty open subset $U\subset Y \setminus ( f(\Sigma_1)\cup  f(\Sigma_2))$ such that
$$\tilde F_y:=\tilde f^{-1}(y)=\pi^{-1}(F_y)\cong F_y\cong \BP^1$$ 
and 
$$G_y:=g^{-1}(y)= u_X^{-1}(F_y)\cong F_y\cong\BP^1$$ 
for all $y\in U(k)$.

Since $Y$ is non-uniruled,  $g(\tilde F_y)$ and $\tilde f(G_y)$ are  points for every $y\in U(k)$  (see, e.g. \cite[Chapter IV, Proposition 1.3, (1.3.4), p. 183]{Kollar}). 

  It follows from   Kawamata's Lemma \cite[Lemma 10.7 on pp. 314--315]{Itaka} applied (twice) to morphisms
$$\tilde f, g: \tilde{X}\to Y$$
 that there exist rational maps $h_1,h_2: Y \dasharrow Y$ such that
$$g=h_1\circ \tilde f, \ \tilde f=h_2 \circ g.$$
This implies  that $h_1$ and $h_2$ are mutually inverse birational automorphisms of $Y$. Let us put
$$u_Y:=h_1\in \Bir(Y).$$
Then $u_Y$ may be 
 included into the commutative digram

\begin{equation}\label{diagram106}\begin{aligned}
& &   &\tilde X       &  {}  & {}        &   && \notag \\
& &\pi& \downarrow  &        {\searrow}                     &  \tilde {u_X}  & &&\notag \\
& & &X       &\stackrel{u_X}{\dashrightarrow } &X  & &&  \\ 
& &f & \downarrow  &  {}&\downarrow  f& &&\notag \\
& &  &Y           &\stackrel {u_Y}{\dashrightarrow} &Y  &&    &  
\end{aligned}.  \end{equation}

For the corresponding embeddings  $k(Y) \hookrightarrow k(X)$ of fields   of rational functions we have:  $f^*\circ (u_Y)^*=\sigma\circ f^*,$  thus 

$$\sigma (k(Y))=k(Y).$$\end{proof}

\begin{rem} Lemma  \ref{Kaw} follows from the  Theorem   
on  {\sl Maximal Rational Connected Fibrations} \cite[Chapter IV, Theorem 5.5, p.223]{Kollar}. However,  our particular case is much easier, 
so we  were tempted to provide a simple proof
rather than to use the powerful theory.

\end{rem}







The next statement follows immediately from  Lemma \ref{Kaw}.

\begin{cor}
\label{exact}
Keeping the notation and assumptions of Lemma \ref{Kaw}, the map 
$$u_X \mapsto u_Y$$
 gives rise to the group homomorphism
$\Bir(X) \to \Bir(Y)$, whose kernel is 
$$\Aut(k(X)/k(X))=\Bir_{k(Y)}(\mathcal{X}_f)$$
 (see \eqref{kernel}) where 
$\mathcal{X}_f$ is the generic fiber of $f$.
 In particular, we get an exact sequence of groups
$$\{1\} \to \Bir_{k(Y)}(\mathcal{X}_f)  \hookrightarrow  \Bir(X) \to \Bir(Y).$$
\end{cor}

\begin{rem}
The special case of Corollary \ref{exact} when $Y$ is an abelian surface may be deduced from \cite[Cor. 1.7]{Sarkisov1}.
\end{rem}

\begin{cor}
\label{exactJ}
Keeping the notation and assumptions of Lemma \ref{Kaw}, suppose that 
 $\Bir(\mathcal{X}_f)$ is bounded.
Then $\Bir(X)$ is strongly Jordan.
\end{cor}

\begin{proof}
By results of Prokhorov and Shramov (Theorem \ref{PS} and Remark \ref{PSr}),  $\Bir(Y)$ is strongly Jordan, because $Y$ is {\sl non-uniruled}.
More precisely,  they proved in  \cite[Corllary 6.8]{ProSha1} that  group $\Bir(Y)$ is Jordan provided $Y$ is non-uniruled. In  \cite[Remark  6.9]{ProSha1}  they claim that actually if $Y$ is non-uniruled then
$\Bir(Y)$ has  finite subgroups of bounded rank (and therefore is strongly Jordan): the proof has to be  based on the same arguments as the proof of Cor. 6.8 of \cite{ProSha1}  but is not presented.  Let us provide the needed mild modifications of the proof of Cor. 6.8 of \cite{ProSha1}  in order to obtain that non-uniruledness of $Y$ implies having finite subgroups of bounded rank.
Indeed, for  $Y$  non-uniruled  
and   for a finite commutative subgroup $G\subset \Bir(Y)$ one has (using  \cite[Proposition 6.2]{ProSha1} and its notation except replacing $X$ by $Y$ and $X_{\mathrm{nu}}$ by $Y_{\mathrm{nu}}$)  the following.
\begin{itemize}\item

$Y_{\mathrm{nu}}\sim Y$ (thanks to Lemma 4.4 and Remark 4.5 of \cite[Lemma 4.4 on p. 2061]{ProSha1}) and  the group $G_{\mathrm{nu}}$ is isomorphic to $G$. In particular,  $G_{\mathrm{nu}}$ is also finite and commutative. 
\item
 There are short exact sequences (\cite[p. 2063, 6.5]{ProSha1})
$$1 \to G_{alg} \to G \to G_N \to 1,$$
$$1\to G_L \to G_{alg} \to G_{ab} \to 1,$$
where 
\begin{enumerate}
\item[(1)]
The finite groups
$G_{\mathrm{alg}}$ and $G_N$ are commutative.
\item[(2)]  $\rank(G_{\mathrm{ab}})\le m_1$ where $m_1$ depends only on $q(Y)$.
\item[(3)] $\mid G_L\mid \le n=n(Y),$ i.e. $G_L $ is finite  and its order is bounded from above by a number $n$ that depends  on $Y$ but not on $G$. (This follows from Lemma 5.2 of \cite[p. 2062]{ProSha1}.)
\item[(4)] It follows from (2) and (3) combined with the  exact sequence (6.5) of \cite[p. 2063]{ProSha1} that $\rank(G_{\mathrm{alg}})\le n + m_1$.
\item[(5)]  $\mid G_N\mid\le b:=b(Y)$ is bounded from above by a number $b(Y)$ that depends only on $Y$. (It follows from Cor. 2.14 of \cite[p. 2058]{ProSha1}.)
\end{enumerate}\end{itemize}

It follows from the exact sequences    that  
$$\rank(G)=\rank(G_{\mathrm{nu}})\le \rank(G_{\mathrm{alg}})+b(Y)\le (n + m_1)+b(Y) =:m(Y)$$
where the  bound $m(Y)$  depends on $Y$ but does not depend on $G$. This proves that $\Bir(Y)$ has finite subgroups of  bounded rank.

So, we know that $\Bir(Y)$ is strongly Jordan.
Now the desired result follows readily from Corollary \ref{exact} combined with   Lemma \ref{klyachko2}.
\end{proof}

In the next section we prove that  $\Bir(\mathcal{X}_f)$ is bounded  if $\mathcal{X}_f$ is {\sl not} the projective line over $k(Y)$.

\section{Linear Algebra}\label{4}
\label{linalg}
Throughout this section $K$ is a field of characteristic $0$ that contains {\sl all} roots of unity. Let $V$ be a vector space over $K$ of finite positive dimension $d$.
We write $1_V$ for the identity automorphism of $V$.  As usual, $\End_K(V)$ stands for the algebra of $K$-linear operators in $V$ and
$$\Aut_K(V)=\End_K(V)^{*}$$
for the group of linear invertible operators in $V$.  We write
$$\det={\det}_V:\End_K(V) \to K$$
for the determinant map. It is well known that $\Aut_K(V)$ consists of all elements of $\End_K(V)$ with {\sl nonzero} determinant and
$$\det: \Aut_K(V) \to K^{*}$$
is a group homomorphism.

 Since $K$  has characteristic zero and contains all roots of unity,  every
 periodic ($K$-linear) automorphism
  $u \in \Aut_K(V)$  of $V$  admits a  basis of $V$ that consists of eigenvectors of $u$, because $u$ is semisimple and all its eigenvalues  lie in  $K$.

Let 
$$\phi: V \times V \to K$$
be a nondegenerate  symmetric $K$-bilinear form that is {\sl anisotropic}, i.e.
$\phi(v,v)\ne 0$ for all {\sl nonzero} $v \in V.$ The form $\phi$ defines the {\sl involution of the first kind}
$$\sigma=\sigma_{\phi}: \End_K(V) \to \End_K(V)$$
characterized by the property
$$\phi(ux,y)=\phi(x, \sigma(u)y) \ \forall x,y \in V$$
(see \cite[Ch. 1]{Involutions}).  It is known \cite[Ch. 1, Sect. 2, Cor. 2.2 on p.14 and Prop. 2.19 on p. 24]{Involutions} that
$$\det(u)=\det(\sigma(u)) \ \forall u \in \End_K(V).$$

 We write
$\GO(V,\phi)\subset \Aut_K(V)$ for the (sub)group of {\sl similitudes} of $\phi$. In other words,  a $K$-linear automorphism $u$ of $V$ lies in $\GO(V,\phi)$ if and only if there exists 
$$\mu=\mu(g) \in K^{*}$$ such that
$$\phi(ux,uy)=\mu \cdot \phi(x,y) \ \forall x,y \in V.$$
If this is the case then
$$\sigma(u) u=\mu\cdot 1_V.$$
Clearly,
$$K^{*}\cdot 1_V \subset \GO(V,\phi).$$
 We have
$$\SO(V,\phi)\subset \OO(V,\phi) \subset \GO(V,\phi)$$
where
$$\OO(V,\phi) =\{u \in \Aut_K(V)\mid \phi(ux,uy)= \phi(x,y) \ \forall x,y \in V\}$$ 
while $\SO(V,\phi)$ consists of all elements of $\OO(V,\phi)$ with determinant $1$. (Recall that elements of $\OO(V,\phi)$  have determinant $1$ or $-1$. In particular, 
$\SO(V,\phi)$ is a normal subgroup of index $2$ in $\OO(V,\phi)$.) 
Clearly,
$$\OO(V,\phi) =\{u \in \Aut_K(V)\mid   \sigma(u) u=1_V\}.$$
It is also clear that
$$\GO(V,\phi) \to K^{*}, \ u \mapsto \mu(u)$$
is a group homomorphism, whose kernel coincides with
$\OO(V,\phi)$; in particular, $\OO(V,\phi)$ is a {\sl normal} subgroup of $\GO(V,\phi)$.
It is well known (and  may be easily checked) that
$$\OO(V,\phi)\bigcap [K^{*} \cdot 1_V]=\{\pm 1_V\};$$
in addition, if $d=\dim(V)$ is {\sl odd} then
$$\SO(V,\phi)\bigcap [K^{*} \cdot 1_V]=\{1_V\}.$$
We denote by $\PO(V,\phi)$ the quotient group $\GO(V,\phi)/(K^{*} \cdot 1_V)$. 

\begin{rem}
\label{import}
The importance of the group $\PO(V,\phi)$ is explained by the following result \cite[Sect. 69,  Corollary 69.6 on p. 310]{KarpenkoM}.
Let  
$$q(v):=\phi(v,v)$$
be the corresponding quadratic form on $V$ and let
 $$X_{q}\subset \mathbb{P}(V)$$
 be  the projective quadric
   defined by the equation $q(v)=0$, which is a smooth projective irreducible  $(d-2)$-dimensional variety over $K$. Then  the groups $\Aut(X_{q})$ and $\PO(V,\phi)$ are isomorphic. 
\end{rem}

\begin{rem}
\label{karp}
\label{surO}
Restricting the surjection
$$\GO(V,\phi) \twoheadrightarrow \GO(V,\phi)/(K^{*} \cdot 1_V)=\PO(V,\phi)$$
 to the subgroup $\OO(V,\phi)$, we get a group homomorphism
$$\OO(V,\phi) \to \PO(V,\phi),$$
whose kernel is   a
 finite   subgroup $\{\pm 1_V\}$. This implies that if $u$ is an element of $\OO(V,\phi)$, whose image in $\PO(V,\phi)$ has finite order then $u$ itself has finite order.
 \end{rem}

\begin{lem}
\label{order2}
Let $u$ be an element of finite order in $\OO(V,\phi)$.  Then $u^2=1_V$.
\end{lem}

\begin{proof}
Let $\lambda$ be an eigenvalue of $u$. Then $\lambda$ is a root of unity and therefore lies in $K$. This implies that there is a (nonzero) eigenvector $x \in V$ with $ux=\lambda x$. Since $u \in \OO(V,\phi)$,
$$\phi(ux,ux)=\phi(x,x).$$
Since $ux=\lambda x$,
$$\phi(ux,ux)=\phi(\lambda x,\lambda x)=\lambda^2 \phi(x,x)$$
and therefore
$\lambda^2 \phi(x,x)=\phi(x,x)$.
Since $\phi$ is anisotropic, $\phi(x,x)\ne 0$ and therefore $\lambda^2=1$.  In other words, every eigenvalue of $u^2$ is $1$ and therefore (semisimple) $u^2=1_V$.
\end{proof}

\begin{cor}
\label{exp2}
Let $G$ be a finite subgroup of $\OO(V,\phi)$. If $G$ does not coincide with $\{1_V\}$ then it is a commutative group of exponent $2$, whose order divides $2^d$. If, in addition, $G$ lies in $\SO(V,\phi)$ then its order divides $2^{d-1}$.
\end{cor}

\begin{proof}
By Lemma \ref{order2}, every $u \in G$ satisfies $u^2=1_V$. This implies that $G$ is commutative.  In addition, $G$ is a $2$-group, i.e., its order is a power of $2$. The commutativeness of $G$ implies that there is a basis of $V$ such that the matrices of all elements of $G$ become diagonal with respect to this basis. Since all the diagonal entries are either $1$ or $-1$,  the order of $G$ does not exceed $2^d$ and therefore divides $2^d$. If, in addition, all elements of $G$ have determinant $1$ then the order of $G$ does not exceed $2^{d-1}$ and, therefore, divides $2^{d-1}$. 
\end{proof}


\begin{cor}
\label{exp4}
Let $\mathbf{u}$ be an element of finite order in $\PO(V,\phi)$. Then $\mathbf{u}^4=1$.
\end{cor}

\begin{proof}
 Choose an element
 $u$   of $\GO(V,\phi)$ such that its image in  $\PO(V,\phi)$ coincides with $\mathbf{u}$. Then there is $\mu \in K^{*}$ such that
$$\phi(ux,uy)=\mu \phi(x,y) \ \forall \ x,y \in V.$$
This implies that $u_2:=\mu^{-1} u^2$ lies in $\OO(V,\phi)$. Clearly, the image $\bar{u}_2 \in \PO(V,\phi)$ of $u_2$ coincides with $\mathbf{u}^2$ and therefore has finite order. By Corollary \ref{surO}.
 $u_2$ has finite order. It follows from Lemma \ref{order2} that $u_2^2=1_V$. This implies that $\mathbf{u}^2$ has order $1$ or $2$ and therefore the order of  $\mathbf{u}$  divides $4$.
\end{proof}

\begin{cor}
\label{nilp}
If $\Bc$ is a finite subgroup of $\PO(V,\phi)$ then it sits in a short exact sequence
$$\{1\} \to \Ac_1 \to \Bc \to \Ac_2 \to \{1\}$$
where both $\Ac_1$ and $\Ac_2$ are finite elementary commutative $2$-groups and  $|\Ac_1|$ divides $2^{d-1}$. In particular, each finite subgroup $\Bc$ of $\PO(V,\phi)$ is a finite $2$-group such that
$$[[\Bc,\Bc], [\Bc,\Bc]]=\{1\}.$$
\end{cor}

\begin{proof}
Let $\Ac_1$ be the subgroup of all elements of $\Bc$ that are the images of elements of $\OO(V,\phi)$. Since $\OO(V,\phi)$ is normal in $\GO(V,\phi)$, the subgroup $\Ac_1$ is normal in $\Bc$. It follows from the proof of Corollary \ref{exp4} that for each  $\mathbf{u}\in \Bc$ its square $\mathbf{u}^2$ lies in $\Ac_1$. This implies that all the elements of the quotient
$\Ac_2:=\Bc/\Ac_1$ have order $1$ or $2$. It follows that $\Ac_2$ is an elementary abelian $2$-group.  We get a short exact sequence
$$\{1\} \to \Ac_1 \to \Bc \to \Ac_2 \to \{1\}.$$
Let $\tilde{\Ac}_1$ be the preimage of $\Ac_1$ in $\OO(V,\phi)$. 
Clearly,  $\tilde{\Ac}_1$ is a subgroup of $\OO(V,\phi)$ and 
$|\tilde{\Ac}_1|=2\cdot |\Ac_1|$. On the other hand, it follows from Corollary \ref{exp2} that $\tilde{\Ac}_1$ is an elementary abelian $2$-group, whose order divides $2^d$. Since $\tilde{\Ac}_1$ maps {\sl onto} $\Ac_1$, the latter is also an elementary abelian $2$-group and its order divides $\frac{1}{2}2^d=2^{d-1}$. Since
$$|\Bc|=|\Ac_1| \cdot |\Ac_2|,$$
the order of $\Bc$ is a power of $2$, i.e., $\Bc$ is a finite $2$-group. On the other hand, since $\Ac_2=\Bc/\Ac_1$ is abelian, $[\Bc,\Bc]\subset \Ac_1$. Since $\Ac_1$ is abelian,
$$[[\Bc,\Bc], [\Bc,\Bc]]\subset [\Ac_1,\Ac_1]=\{1\},$$
i.e., $[[\Bc,\Bc], [\Bc,\Bc]]=\{1\}$.
\end{proof}

In the case of {\sl odd} $d$ we can do better. Let us start with the following observation.

\begin{lem}
\label{odd}
Suppose that $d=2\ell+1$ is an odd integer that is greater or equal than $3$. Then every $u \in \GO(V,\phi)$ can be presented as
$$u =\mu_0 \cdot u_0$$
with $u_0 \in \SO(V,\phi)$ and $\mu_0 \in K^{*}$
\end{lem}

\begin{ex}
If $u$ is an element of $\OO(V,\phi)$ with determinant $-1$ then 
$$u=(-1)\cdot (-u), \ (-u) \in \SO(V,\phi).$$
\end{ex}

\begin{proof}[Proof of Lemma \ref{odd}]
Recall that there is $\mu \in K^{*}$ such that
$$\phi(ux,uy)=\mu \cdot \phi(x,y) \ \forall x,y \in V$$ 
and therefore
$$\sigma(u) u=\mu \cdot 1_V.$$
Now let $\gamma \in K^{*}$ be the determinant of $u$. Since
$\det(\sigma(u))=\det(u)$, we obtain that
$$\mu^{2\ell+1}=\mu^{d}=\det(\mu\cdot 1_V)=\det(\sigma(u) u)=\det(\sigma(u))\det(u)=\gamma^2.$$
This implies that
$$\gamma^2=\mu^{d}=\mu^{2\ell+1}.$$
Let us put 
$$\mu_0=\frac{\gamma}{\mu^{\ell}}, \ u_0=\mu_0^{-1} \cdot u.$$
Then
$$\mu_0^2=\mu, \ \gamma=\mu_0^{2\ell+1}=\mu_0^d,$$
$$u=\mu_0\cdot u_0, \ \det(u_0)=\mu_0^{-d}\cdot\det(u)=\gamma^{-1}\gamma=1.$$
We also have
$$\phi(u_0 x, u_0 y)=\phi(\mu_0^{-1}u x, \mu_0^{-1}u y)=
\mu_0^{-2}\cdot \phi(u x, u y)=\mu^{-1} \phi(u x, u y)=
\mu^{-1}\cdot \mu \cdot \phi(x,y)=\phi(x,y).$$
This implies that $u_0 \in \OO(V,\phi)$. Since $\det(u_0)=1$,
$$u_0 \in \SO(V,\phi).$$
\end{proof}

\begin{cor}
\label{oddPO}
Suppose that $d=2\ell+1$ is an odd integer that is greater or equal than $3$. Then the group homomorphism
$$\mathrm{prod}: K^{*} 1_V \times  \SO(V,\phi) \to \GO(V,\phi),
\  (\mu_0\cdot 1_V, u_0) \mapsto \mu_0 \cdot u_0$$
is a group isomorphism. In particular, the group
$\PO(V,\phi)=\GO(V,\phi)/(K^{*} 1_V)$ is canonically isomorphic to $\SO(V,\phi) $.
\end{cor}

\begin{proof}Since $d$ is odd, 
$$\SO(V,\phi)\bigcap [K^{*} \cdot 1_V]=\{1_V\},$$
which implies that $\mathrm{prod}$ is injective. Its surjectiveness follows from Lemma \ref{odd}.
\end{proof}

\begin{thm}
\label{GPOV}
Suppose that $K$ is a field of characteristic zero that contains all roots of unity,  $d \ge 3$ is an odd integer,  $V$ is a $d$-dimensional $K$-vector space and
$$\phi: V \times V \to K$$
a nondegenerate  symmetric $K$-bilinear form that is  anisotropic, i.e.
$\phi(v,v)\ne 0$ for all {\sl nonzero} $v \in V.$

Let $G$ be a finite subgroup in $\PO(V,\phi)$. Then $G$ is commutative, all its non-identity elements have order $2$ and the order of $G$ divides $2^{d-1}$.
\end{thm}

\begin{proof}
The result follows readily from Corollary \ref{oddPO} combined with
Corollary \ref{exp2}.
\end{proof}


\begin{cor}
\label{quadric}
Suppose that $K$ is a field of characteristic zero that contains all roots of unity,  $d \ge 3$ an odd integer,  $V$ a $d$-dimensional $K$-vector space and let $q: V \to K$ be a  quadratic form such that $q(v) \ne 0$ for all nonzero $v \in V$.
Let us consider  the projective quadric $X_{q}\subset \mathbb{P}(V)$  defined by the equation $q=0$, which is a smooth projective irreducible  $(d-2)$-dimensional variety over $K$.
Let $\Aut(X_{q})$ be the group of biregular automorphisms of $X_{q}$.
Let $G$ be a finite subgroup in $\Aut(X_{q})$. Then $G$ is commutative, all its non-identity elements have order $2$ and the order of $G$ divides $2^{d-1}$.
\end{cor}

\begin{proof}
Let $\phi: V \times V \to K$ be the symmetric $K$-bilinear form such that
$$\phi(v,v)=q(v) \ \forall v \in V.$$
Namely, for all $x,y\in V$
$$\phi(x,y):=\frac{q(x+y)-q(x)-q(y)}{2} \ .$$
Clearly, $\phi$ is nondegenerate.   In the notation of \cite[Sect. 69, p. 310]{KarpenkoM},
$$\GO(q)=\GO(V,\phi), \  \PO(q)=\PO(V,\phi).$$
By Corollary 69.6 of \cite[Sect. 69]{KarpenkoM}, 
 the groups $\Aut(X_{q})$ and $\PO(q)$ are isomorphic. Now the result follows from Theorem \ref{GPOV}.
\end{proof}

\begin{cor}
\label{genuszero}

Suppose that $K$ is a field of characteristic zero that contains all roots of unity. Let $\mathcal{C}$ be a smooth irreducible projective genus $0$ curve over $K$ that is not biregular to $\mathbb{P}^1$ over $K$. 

Let $\Bir_K(\mathcal{C})$ be the group of birational automorphisms of $\mathcal{C}$.
Let $G$ be a finite subgroup in $\Bir_K(\mathcal{C})$. Then $G$ is commutative, all its non-identity elements have order $2$ and the order of $G$ divides $4$. In other words, if $G$ is nontrivial then it is either a cyclic group of order $2$ or is isomorphic to a product of two cyclic groups of order $2$.
\end{cor}

\begin{proof}
Since $\mathcal{C}$ has genus zero, it is $K$-biregular to a smooth projective plane quadric
$$\mathcal{X}=\{a_1 T_1^2 +a_2 T_2^2+a_3 T_3^2=0\}\subset \mathbb{P}^2$$
where all $a_i$ are {\sl nonzero} elements of $K$.  Since $\mathcal{C}$  is {\sl not} biregular to $\mathbb{P}^1$,  the set $\mathcal{X}(K)$ is {\sl empty} (\cite[Th.  A.4.3.1 on p. 75]{HindryS}, \cite[Sect. 45, Prop. 45.1 on p. 194]{KarpenkoM}), which means that the nondegenerate ternary quadratic form
$$q(T)=a_1 T_1^2 +a_2 T_2^2+a_3 T_3^2$$
is {\sl anisotropic}.  We may view $q$ as the quadratic form on the coordinate three-dimensional $K$-vector space $V=K^3$.  Then (in the notation of Corollary \ref{quadric})  $d=3$ and $\mathcal{X}=X_q$.   Since $\mathcal{X}$ is a smooth projective curve, its group  $\Bir_K(\mathcal{X})$ of birational automorphisms coincides with the group  $\Aut(\mathcal{X})$ of biregular automorphisms.
Now the result follows from Corollary \ref{quadric}.
\end{proof}

The rest of this section deals with the case of {\sl even} $d$; its results will not be used elsewhere in the paper.

\begin{thm}
\label{AllDim}

Suppose that $K$ is a field of characteristic zero that contains all roots of unity,  $d\ge 2$ is an even positive  integer,  $V$ is a $d$-dimensional $K$-vector space, and
$$\phi: V \times V \to K$$
a nondegenerate  symmetric $K$-bilinear form that is  anisotropic, i.e.
$\phi(v,v)\ne 0$ for all {\sl nonzero} $v \in V.$
Then the group  $\PO(V,\phi)$ is bounded. More precisely, there is a positive integer $\mathbf{n}=\mathbf{n}(d)$ that depends only on $d$ and such 
that every finite subgroup of $\PO(V,\phi)$  has order dividing $2^{ \mathbf{n}(d)}$.
\end{thm}

\begin{proof}
We deduce  Theorem \ref{AllDim} from Theorem  \ref{boundA}.
Let $\Aut_K(\End_K(V))$ be the group of automorphisms of the $K$-algebra $\End_K(V)$. We write $\mathcal{V}_2$ for $\End_K(V)$ viewed as the $d^2$-dimensional $K$-vector space and $\Aut_K(\mathcal{V}_2)$ for its group of $K$-linear automorphisms. We have
$$\Aut_K(\End_K(V))\subset \Aut_K(\mathcal{V}_2).$$
Let us choose a basis $\{e_1, \dots, e_{d^2}\}$ of  $\mathcal{V}_2$. Such a choice gives us a group isomorphism
$$\Aut_K(\mathcal{V}_2)\cong  \GL(d^2,K).$$
 Let us consider a group homomorphism
$$\mathrm{Ad}: \GO(V,\phi)\subset \Aut_K(V) \to \Aut_K(\End_K(V)), \ u \mapsto \{w \mapsto uwu^{-1} \ \forall \ w\in \End_K(V)\}$$
for all $u\in \GO(V,\phi)\subset \Aut_K(V)$. Clearly, 
$$\ker(\mathrm{Ad})=K^{*}\cdot 1_V.$$
 This gives us an embedding
$$\PO(V,\phi)=\GO(V,\phi)/\{K^{*}\cdot 1_V\} \hookrightarrow \Aut_K(\End_K(V))\hookrightarrow \Aut_K(\mathcal{V}_2)\cong  \GL(d^2,K).$$
This implies that $\PO(V,\phi)$ is isomorphic to a subgroup of $\GL(d^2,K)$.  By Corollary \ref{exp4}, every periodic element 
in $\PO(V,\phi)$ has order 
dividing $4$. This implies (thanks to First Sylow Theorem) that  the order of every finite subgroup of $\PO(V,\phi)$ is a power of $2$. In other words, all finite subgroups in $\PO(V,\phi)$ are $2$-groups.
 Now 
the desired result follows from Theorem  \ref{boundA} (applied to $n=d^2$ and $N=4$).
\end{proof}

 Combining Theorem \ref{AllDim}, Corollary \ref{nilp}  and Remark \ref{karp}, we obtain the following assertion.

\begin{thm}
Suppose that $K$ is a field of characteristic zero that contains all roots of unity,  $d\ge 2$  an even integer,  $V$ a $d$-dimensional $K$-vector space. Let $q: V \to K$ be a  quadratic form such that $q(v) \ne 0$ for all nonzero $v \in V$.
Let us consider  the projective quadric $X_{q}\subset \mathbb{P}(V)$  defined by the equation $q=0$, which is a smooth projective irreducible  $(d-2)$-dimensional variety over $K$.
Let $\Aut(X_{q})$ be the group of biregular automorphisms of $X_{q}$.

Then:
\begin{enumerate}
\item
 $\Aut(X_{q})$ is bounded.  More precisely, there is a positive integer $\mathbf{n}=\mathbf{n}(d)$ that depends only on $d$ and such 
that every finite subgroup of $\Aut(X_{q})$  has order  dividing $2^{ \mathbf{n}(d)}$.
\item
If $\Bc$ is a finite subgroup of $\Aut(X_{q})$ then it  is a finite $2$-group that
 sits in a short exact sequence
$$\{1\} \to \Ac_1 \to \Bc \to \Ac_2\to \{1\}$$
where both $\Ac_1$ and $\Ac_2$ are finite elementary  abelian $2$-groups and  $|\Ac_1|$ divides $2^{d-1}$. In particular, 
$$[[\Bc,\Bc],[\Bc,\Bc]]=\{1\}.$$
\end{enumerate}
\end{thm}

We hope to return to a classification of finite subgroups of $\Aut(X_{q})$ (for even $d$) in a future publication.

\section{Jordaness properties of $\Bir$}\label{5}

\begin{proof}[Proof of Theorem \ref{BirXY}]
Let us put $K=k(Y)$. Then $\fchar(K)=0$. Since $K$ contains algebraically closed $k$, it contains all roots of unity.  In the notation of Corollary \ref{genuszero} let us put $\mathcal{C}=\mathcal{X}_f$. Since $\mathcal{X}_f$ has no $K$-points, it  is {\sl not} biregular to $\mathbb{P}^1$ over $K$.

It follows from Corollary \ref{genuszero} that $\Bir(\mathcal{X}_f)$ is {\sl bounded}. Now Corollary \ref{exactJ} implies that $\Bir(X)$ is strongly Jordan.
\end{proof}
\begin{proof}[Proof of Theorem \ref{BirXYbis}]
It follows from Remark \ref{generic-general} that $f: X \to Y$ is a conic bundle. In particular, the generic fiber $\mathcal{X}=\mathcal{X}_f$ is an absolutely irreducible smooth projective genus zero curve over $K:=k(Y)$.  

 Since each $K$-point of $\mathcal{X}$ gives rise to a rational section $Y \dasharrow X$ of $f$, there are no $K$-points on $\mathcal{X}$.
  Now the result follows from Theorem \ref{BirXY}.
\end{proof}

\begin{ex}
Let $Y$ be a smooth  irreducible projective variety over $k$ of dimension $\ge 2$.  Let $a_1, a_2,a_3$ be nonzero elements of $k(Y)$ such  that the ternary quadratic form
$$q(T)=a_1 T_1^2 +a_2 T_2^2 +a_3 T_3^2$$
is {\sl anisotropic} over $k(Y)$.  Example \ref{aniquad}
gives us a a smooth  irreducible projective variety $\tilde{X}_q$ and a surjective regular map $f:\tilde{X}_q\to Y$, whose generic fiber is the quadric
$$\{a_1 T_1^2 +a_2 T_2^2+a_3 T_3^2=0\}\subset \mathbb{P}^2_{k(Y)}$$
over $k(Y)$ without $k(Y)$-points. Now Theorem \ref{BirXY} tells us that $\Bir(\tilde{X}_q)$ is {\sl strongly Jordan} if $Y$ is non-uniruled.
\end{ex}

\begin{rem}
Recall  (Example \ref{aniquad}) that if $\dim(Y)\ge 2$ then there always exists an {\sl anisotropic} ternary quadratic form over $k(Y)$. (A theorem of Tsen implies that such a form does {\sl not} exists if $\dim(Y)=1$.)
\end{rem}

\begin{proof}[Proof of Theorem \ref{BirX}]
Recall that an abelian variety $Y$ does {\sl not} contain rational curves; in particular, it is {\sl not} uniruled. Now 
Theorem \ref{BirX} follows from Theorem \ref{BirXY}.
\end{proof}

\begin{thm}
\label{uniformA}
Let $d \ge 3$ be an integer. Let $\mathcal{G}$ be the collection of groups $\Bir(X)$ where $X$ runs through the set of  smooth irreducible projective $d$-dimensional varieties that can be realized as conic bundles $f:X \to Y$ over a $(d-1)$-dimensional abelian variety $Y$ but $X$ is not birational to
 $Y \times \mathbb{P}^1$.
Then $\mathcal{G}$  is uniformly strongly Jordan.
\end{thm}

\begin{proof}
It follows from Remark \ref{productYP} that  the generic fiber $\mathcal{X}_f$ of $f$ has {\sl no} $k(Y)$-rational points. 
It follows from Corollary \ref{genuszero} that the collection of groups of the form $\Bir(\mathcal{X}_f)$ is uniformly bounded  - actually,  the order of every finite subgroup in $\Bir(\mathcal{X}_f)$ divides $4$.  Recall (Corollary \ref{exact}) that there is an exact sequence
$$\{1\} \to \Bir(\mathcal{X}_f) \hookrightarrow \Bir(X) \to \Bir(Y).$$
Now the result follows from Corollary \ref{abelianUniform}.
\end{proof}

\begin{rem} 
\label{singularConics}
 The condition that $k$-varieties $X,Y$ in  \thmref{BirX}, \thmref{BirXY}, and  \thmref{BirXYbis}
are smooth, is non-essential.  Indeed, let $X,Y$ be irreducible projective varieties  of dimensions $d$ and $d-1$ , respectively, endowed with  a surjective morphism 
$f:X\to Y.$ 
Let $Y$ be non-uniruled. Due to the resolution of singularities (see, for example, \cite[Chapter 3, section 3.3]{Kol1})  one  can always 
find two smooth projective irreducible varieties $\tilde X, \tilde Y,$   the   birational  morphisms $\pi_X:\tilde X\to X,$ \ $\pi_Y:\tilde Y\to Y,$ \ and  a morphism
$\tilde f :\tilde X\to \tilde Y$  such that  the following diagram is commutative:
\begin{equation}\label{diagram100}
\begin{CD}
\tilde X   @>{\tilde f}>>  \tilde Y\\
@V\pi_X VV @V\pi_Y VV \\
X @>{f}>> Y
\end{CD}.\end{equation}

Then the  following properties are valid:
\begin{enumerate}\item[1.]    $\tilde Y$  is non-uniruled,  since it is birational to the non-uniruled $Y$
(see, for example, \cite[Chapter 4, Remark 4.2] {Deba}).
\item[2.]   If the  general fiber  $F_u:=f^{-1}(u), \ u\in U\subset Y$ is  irreducible, then so is the general  fiber
 $\tilde F_v:=\tilde f^{-1}(v),  \ v\in V\subset \tilde Y$  of $\tilde f$ (here $U, V$ 
are open  dense  subsets of $Y,  \ \tilde Y,$  respectively). 

Indeed, there is an open dense  $V'\subset \tilde Y$ such that

--$\pi_Y$ is an isomorphism of $V'$ to $\pi_Y(V');$

--for every $v\in V'$
the exceptional set $S_X$ of morphism $\pi_X$  intersects  the 
fiber $\tilde F_v$ of $\tilde f$ only at  an empty or a finite set of points.  It is valid, because  $  \dim (S_X)\le \dim (\tilde Y)= d-1, $   hence   the restriction of $\tilde f$ onto an irreducible component of $S_X$  is either non-dominant, or generically finite.

\noindent Thus, 
$\tilde F_v \cap(\tilde X\setminus S_X)$ is open  and dense in  $\tilde F_v$ for  point $v\in V'$ (because every irreducible component  of $\tilde F_v$ has dimension at least $1$).
On the other hand,   $\tilde F_v\cap(\tilde X\setminus S_X)$  is isomorphic via $\pi_X$  to  $F_{\pi_Y(v)} \cap( X\setminus \pi_X(S_X)),$   which is an open and dense subset of   irreducible $F_{\pi_Y(v)}, $  if $\pi_Y(v)\in U.$
Hence, for all $v\in V'\cap \pi_Y^{-1}(U)$ the fibers $\tilde F_v$ and $F_{\pi_Y(v)}$  are birational.

\item [3.]  If the  generic fiber $\mathcal  X_{f}$ of $f$ is absolutely irreducible, so is  generic  the  fiber $\mathcal X_{\tilde  f}$  of $\tilde f.$  Indeed, in this case the general fiber of $f$ is irreducible  (see
\cite[Proposition 9.7.8] {EGA43}). According to property [2], the general fiber of $\tilde f$ is also  irreducible, and, hence, so is $\mathcal X_{\tilde  f}$  ({\it ibid}).

\item [4.]   The general and generic fibers $\tilde F_v$ and  $\mathcal X_{\tilde f}$  of $\tilde f$  are smooth, since ${\tilde f}$ is a surjective morphism between smooth projective varieties.
\end{enumerate}
   
It follows that if the  generic  (respectively, general) fiber of $ f$  is a  rational curve, then 
the  generic  (respectively, general) fiber of $\tilde f$  is a smooth rational curve. 
 According to \thmref{BirXY},   (respectively,  \thmref{BirXYbis} , )  $\Bir(X)=\Bir(\tilde  X)$ has to be Jordan.

\end{rem}


\end{document}